\documentclass[a4paper,11pt,reqno]{amsart}
\usepackage{amssymb,amsmath,hyperref}
\title[The mixed mock modularity of certain duals of Hikami and Lovejoy]{The mixed mock modularity of certain duals of generalized quantum modular forms\\ of Hikami and Lovejoy}
\author{Eric T. Mortenson}
\address{Department of Mathematics and Computer Science, Saint Petersburg State University, Saint Petersburg, 199034, Russia}
\email{etmortenson@gmail.com}

\author{Sander Zwegers}
\address{Department of Mathematics and Computer Science, University of Cologne, Weyertal 86--90, 50931, Cologne, Germany}
\email{szwegers@uni-koeln.de}

\newcommand\N{\mathbb{N}}
\newcommand\Z{\mathbb{Z}}
\newcommand\C{\mathbb{C}}
\renewcommand\theta{\vartheta}
\newcommand\smod[1]{\ (\operatorname{mod} #1)}
\newcommand\sg{\operatorname{sg}}
\newtheorem{theorem}{Theorem}
\newtheorem{lemma}[theorem]{Lemma}
\newtheorem{corollary}[theorem]{Corollary}

\theoremstyle{definition}
\newtheorem{definition}[theorem]{Definition}
\newtheorem{remark}[theorem]{Remark} 
\newtheorem{example}[theorem]{Example}

\numberwithin{theorem}{section} 
\numberwithin{equation}{section}

\begin{document}

\date{6 July 2022}



\begin{abstract}
    We express a family of Hecke--Appell-type sums of Hikami and Lovejoy in terms of mixed mock modular forms; in particular, we express the sums in terms of Appell functions and theta functions.  Hikami and Lovejoy's family of Hecke--Appell-type sums was obtained by considering certain duals of generalized quantum modular forms.
\end{abstract}

\maketitle

\section{Introduction and Statement of the Main Results}
Let $q$ be a nonzero complex number with $|q|<1$.   We recall the $q$-Pochhammer notation:
\begin{equation*}
(x)_n=(x;q)_n:=\prod_{i=0}^{n-1}(1-q^ix), \ \ (x)_{\infty}=(x;q)_{\infty}:=\prod_{i\ge 0}(1-q^ix).
\end{equation*}
We begin with the Kontsevich--Zagier series \cite{Za2001}
\begin{equation*}
F(q):=\sum_{n=0}^{\infty}(q)_n,
\end{equation*}
which is an example of a quantum modular form \cite{Za2020}.  The series $F(q)$ does not converge on any open subset of $\mathbb{C}$, but it is well-defined at any root of unity; in particular, it is a finite sum.

Bryson et al \cite{BOPR} related the Kontsevich--Zagier series to the generating function for strongly unimodal  sequences
\begin{equation*}
U(x;q)=\sum_{n= 0}^{\infty}(-xq)_n(-q/x)_nq^{n+1}.
\end{equation*}
If we denote $\zeta_N:=e^{2\pi i /N}$, they showed that
\begin{equation}
    F(\zeta_{N}^{-1}) = U(-1;\zeta_{N}).\label{equation:id-BOPR}
\end{equation}
In a sense, we can view $U(-1;q)$ as being dual to $F(q).$  Unlike the Kontsevich--Zagier series $F(q)$, the function $U(x;q)$ converges for $|q|<1.$  Identity (\ref{equation:id-BOPR}) can be interpreted in terms of $N$-colored Jones polynomials $J_{N}(K;q)$ for a knot $K$ \cite{HL,WBRL}.

Motivated by the colored Jones polynomial for the torus knot $T_{2,(2t+1)}$ at roots of unity, Hikami obtained a family of quantum modular forms generalizing $F(q)$ \cite{Hik1}, \cite[(1.8)]{HL}:
\begin{equation*}
    F_{t}(q):=q^t\sum_{k_t\ge\cdots\ge k_1\ge 0}(q)_{k_t}
    \prod_{i=1}^{t-1}q^{k_i(k_i+1)}\begin{bmatrix}k_{i+1}\\k_i \end{bmatrix}_q,
\end{equation*}
where $\begin{bmatrix}n\\k \end{bmatrix}_q$ is the $q$-binomial coefficient defined by
\begin{equation*}
 \begin{bmatrix}n\\k \end{bmatrix}_q
 :=\frac{(q)_{n}}{(q)_{k}(q)_{n-k}}.
\end{equation*}

Hikami and Lovejoy then used the perspective of quantum invariants to generalize $U(x;q)$ to a family of $q$-hypergeometric series $U_{t}(x;q)$ \cite[(1.10)]{HL}, 
so that
\cite[Theorem 1.2]{HL}:
\begin{equation}
    F_t(\zeta_{N}^{-1}) = U_t(-1;\zeta_{N}).\label{equation:id-HL}
\end{equation}
Furthermore, Hikami and Lovejoy generalized $F_{t}(q)$ and $U_{t}(x;q)$ to vector-valued forms $F_{t}^{(m)}(q)$ \cite[(5.1)]{HL} and $U_{t}^{(m)}(x;q)$ \cite[(5.19)]{HL}.  In particular, for $1\le m \le t$ they defined
\begin{equation*}
F_{t}^{(m)}(q):=q^t\sum_{k_1,\dots,k_t=0}^{\infty}(q)_{k_t}
q^{k_1^2+\cdots +k_{t-1}^2+k_m+\cdots+k_{t-1}}\prod_{i=1}^{t-1}\cdot
\begin{bmatrix}k_{i+1}+\delta_{i,m-1}\\k_i\end{bmatrix}_{q},
\end{equation*}
where $\delta_{a,b}:=1$ when $a=b$ and $\delta_{a,b}:=0$ otherwise, and 
\begin{align}
U_t^{(m)}(x;q):=q^{-t}\sum_{\substack{k_t\ge \cdots \ge k_1\ge 1\\ k_m\ge 1}}&(-xq)_{k_t-1}(-q/x)_{k_t-1}q^{k_t}\cdot \label{equation:gen-utm}\\
&\cdot \prod_{i=1}^{t-1}q^{k_i^2}\Big [ 
\begin{array} {c} k_{i+1}-k_i-i+\sum_{j=1}^{i}(2k_j+\chi(m>j))\\k_{i+1}-k_i \end{array}
\Big ]_q,\notag 
\end{align}
where $\chi(X):=1$ if $X$ is true and $\chi(X):=0$ otherwise. They then showed that the two families enjoy an identity analogous to (\ref{equation:id-HL}), see \cite[Theorem $5.5$]{HL}:
\begin{equation*}
    F_t^{(m)}(\zeta_{N}^{-1}) = U_t^{(m)}(-1;\zeta_{N}).
\end{equation*}
\begin{remark}  We point out that
\begin{equation*}
F_{t}^{(1)}(q)=F_{t}(q) \ \textup{and} \ U_{t}^{(1)}(x;q)=U_{t}(x;q).
\end{equation*}
\end{remark}

Of course, one is keen to know the (mock) modular properties of $U_{t}(x;q)$ and its vector-valued generalization $U_{t}^{(m)}(x;q). $  To this end, Hikami and Lovejoy also obtained Hecke--Appell-type sums for $U_{t}(x;q)$ and its vector-valued generalization.  They found
\begin{theorem} \cite[Theorem $5.6$]{HL} \label{theo:HL-Theorem5.6} We have
{\allowdisplaybreaks \begin{align*}
U_t^{(m)}&(-x;q) \\
&=-q^{-\frac{t}{2}-\frac{m}{2}+\frac{3}{8}}\frac{(qx)_{\infty}(q/x)_{\infty}}{(q)_{\infty}^2} \notag\\
& \ \ \ \cdot  \Big (\sum_{\substack{r,s\ge 0 \\ r\not\equiv s \pmod{2}}} -\sum_{\substack{r,s < 0 \\ r\not\equiv s \pmod{2}}} \Big )
\frac{(-1)^{\tfrac{r-s-1}{2}}q^{\tfrac{1}{8}r^2+\tfrac{4t+3}{4}rs+\tfrac{1}{8}s^2+\tfrac{1+m+t}{2}r+\tfrac{1-m+t}{2}s}}
{1-xq^{\frac{r+s+1}{2}}}\notag \\
&=-q^{-\frac{t}{2}-\frac{m}{2}+\frac{3}{8}}\frac{(qx)_{\infty}(q/x)_{\infty}}{(q)_{\infty}^2}\notag\\
& \ \ \ \cdot  \Big (\sum_{\substack{r,s,u\ge 0 \\ r\not\equiv s \pmod{2}}} +\sum_{\substack{r,s,u < 0 \\ r\not\equiv s \pmod{2}}} \Big )
(-1)^{\tfrac{r-s-1}{2}}x^uq^{\tfrac{1}{8}r^2+\tfrac{4t+3}{4}rs+\tfrac{1}{8}s^2+\tfrac{1+m+t}{2}r+\tfrac{1-m+t}{2}s+u\tfrac{r+s+1}{2}}.\notag 
\end{align*}}
\end{theorem}

\begin{remark} The case $m=1$ gives \cite[Theorem $1.3$]{HL} as a corollary.
\end{remark}

We recall the standard notation
\begin{equation*}\sg(r,s):=\frac{\sg(r)+\sg(s)}{2}, \ \textup{where} \ 
\sg(r)=\begin{cases} 1 & \text{if}\ r\geq 0,\\ -1 & \text{if}\ r<0.
\end{cases}\qquad 
\end{equation*}
For theta functions, we write
\begin{equation*}
\Theta(x;q):=(x)_{\infty}(q/x)_{\infty}(q)_{\infty}=\sum_{n=-\infty}^{\infty}(-1)^nq^{\binom{n}{2}}x^n.
\end{equation*}

The purpose of this paper is to consider the following functions:
\begin{definition} We define
\begin{equation*}
    g_{t,m}(x):=\sum_{r\not\equiv s\smod{2}} \sg(r,s)\, \frac{(-1)^{\frac{r-s-1}2} q^{\frac18 r^2 + \frac{4t-1}4rs+\frac18 s^2 +\frac{t+m}2r+\frac{t-m}2 s}}{1-xq^{\frac{r+s+1}2}}
\end{equation*}
and
\begin{equation*}
f_{t,m}(x):= -q^{-\frac t2-\frac m2+\frac78}\, \frac{\Theta(x;q)}{(q)_\infty^3}\, g_{t,m}(x)
\end{equation*}
\end{definition}

\begin{remark}
Theorem \ref{theo:HL-Theorem5.6} thus gives
\[ (1-x)\cdot U_t^{(m)}(-x;q)= f_{t+1,m}(x).\]
Note that we suppress $q$ and use $t$ instead of $t+1$ in $g_{t,m}(x)$ and $f_{t,m}(x)$, for convenience.
\end{remark}

To state the main result we need the following indefinite binary theta series.
In Section \ref{section:proof-main} we will actually see that $\theta_{p,m}^*$ is a holomorphic modular form of weight 1.

\begin{definition}\label{def:theta}
Let $t\in\N$ be fixed.
For $p,m\in\Z$ we define
\begin{equation*}\theta_{p,m}:=\sum_{r\not\equiv s\smod{2}} \sg(r,s)\,(-1)^{\frac{r-s-1}2} q^{\frac18 r^2 + \frac{4t-1}4rs+\frac18 s^2 +\frac{p+m}2r+\frac{p-m}2 s}
\end{equation*}
and
\begin{equation*} \theta_{p,m}^*:= q^{-\frac{m^2}{2(2t-1)}+\frac{p^2}{4t} } \theta_{p,m}.
\end{equation*}
\end{definition}

Our main result is the following:

\begin{theorem}\label{theo:main}
For $t\geq 2$ and $1\leq m<t$ we have
\begin{align*}f_{t,m}(x) &= -\frac{q^{-\frac34t-\frac m2+\frac78}}{(q)_\infty^3}\\
& \ \ \ \ \ \cdot \sum_{k\operatorname{mod} 2t} (-1)^k q^{\frac1{4t} (k-t)^2} \theta_{k-t,m} \underset{l\equiv k\smod{2t}}{\sum_{r,l\in\Z}} \sg(r,l)\, (-1)^r q^{\frac12 r^2+lr+\frac{2t-1}{4t}l^2+\frac12 r + \frac12l} x^{-r}.
\end{align*}
\end{theorem}

What is new about Theorem \ref{theo:main} is that it expresses the weight-two object
\begin{equation*}
   (q)_{\infty}^{3} \cdot (1-x) \cdot U_{t}^{(m)}(-x;q) 
\end{equation*}
as a mixed mock modular form \cite{DMZ}. For our purposes, a mixed mock modular form is essentially a finite linear combination of products of theta functions and Appell functions, see Corollary \ref{cor:main}. 

In Section \ref{section:examples}, we will write Theorem \ref{theo:main} in terms of Hecke-type double-sums as studied by the first author and Hickerson \cite[(1.3), (1.4)]{HM}:
\begin{equation}
f_{a,b,c}(x,y;q):=\sum_{r,s}\sg(r,s) (-1)^{r+s}x^ry^sq^{a\binom{r}{2}+brs+c\binom{s}{2}},\label{equation:fabc-def}
\end{equation}
where $a$, $b$, $c$ are positive integers.  In \cite[Theorems 1.3 and 1.4]{HM}, sums of the form (\ref{equation:fabc-def}) were expressed in terms of Appell functions and theta functions.  Here we will use the following (modified) definition of an Appell function \cite{HM}
\begin{equation*}
m(x,z;q):=\frac{1}{\Theta(z;q)}\sum_{r=-\infty}^{\infty}\frac{(-1)^rq^{\binom{r}{2}}z^r}{1-q^{r-1}xz}.
\end{equation*}

The results in \cite{HM} were for general Hecke-type double-sums which involved certain symmetries.  Expansions were found for sums of the form
\begin{equation}
f_{n,n+p,n}(x,y;q) \ \textup{and} \ f_{a,b,c}(x,y;q),\label{equation:fabc-forms}
\end{equation}
where for the first $n$ and $p$ are positive integers, and for the second $a$, $b$, and $c$ are positive integers, both $a$ and $c$ divide $b$, and $b^2-ac>0$.  As an example, a special case of more general results \cite[Theorem 1.3]{HM} reads
\begin{align*}
f_{1,2,1}(x,y;q)
&=\Theta(y;q)m\Big (\frac{q^2x}{y^2},-1;q^3\Big )+\Theta(x;q)m\Big (\frac{q^2y}{x^2},-1;q^3\Big )\\
& \ \ \ \ \ \ \ \ \ \ - y\cdot \frac{(q^3;q^3)_{\infty}^3\Theta(-x/y;q)\Theta(q^2xy;q^3)}
{\Theta(-1;q^3)\Theta(-qy^2/x;q^3)\Theta(-qx^2/y;q^3)}.\notag 
\end{align*}

Unfortunately, some of the double-sums appearing in Theorem \ref{theo:main} (see also Corollary \ref{cor:main}) are not of the two types found in (\ref{equation:fabc-forms}), so the results of \cite[Theorems 1.3, 1.4]{HM} do not apply.  To evaluate such double-sums appearing in Theorem \ref{theo:main}, we will need the following more general result.

\begin{theorem} \label{theo:general-hecke}
Let $a,b,c$ be positive integers such that $D:=b^2-ac>0$.  For generic $x_1$, $x_2$, $y_1$, $y_2$, we have:
\[\begin{split}
\sum_{r_1,s_1\in\Z} &\sg(r_1,s_1)\, q^{\frac12a r_1^2+br_1s_1+\frac12cs_1^2}\, x_1^{r_1}y_1^{s_1}\\
&=\sum_{t=0}^{c-1} x_1^tq^{\frac12at^2} 
\cdot \Theta(-y_1q^{bt+\frac{c}{2}};q^c)
\cdot m(-q^{D(\frac{c}{2}-t)}x_1^cy_1^{-b},-q^{\frac12 cD}x_2;q^{cD})\\
&\qquad + \sum_{t=0}^{a-1}  y_1^t q^{\frac12ct^2}
\cdot \Theta(-x_1q^{bt+\frac{a}{2}};q^{a})
\cdot m(-q^{D(\frac{a}{2}-t)}y_{1}^{a}x_{1}^{-b},-q^{\frac12 aD}y_2;q^{aD})\\
&\qquad -\Big ( \Theta(-x_2q^{\frac12 cD};q^{cD})
\cdot \Theta(-y_2q^{\frac12 aD};q^{aD})\Big ) ^{-1}
\sum_{u,v\smod{b}} q^{\frac12au^2+buv+\frac12cv^2} x_1^{u}y_1^{v} \\
&\qquad\qquad \cdot \sum_{r_2,s_2\in\Z} q^{\frac12 b^2c r_2^2+abc r_2 s_2+\frac12ab^2s_2^2 +c(au+bv)r_2+a(bu+cv)s_2}(x_1^cx_2)^{r_2} (y_1^ay_2)^{s_2}\\
& \qquad\qquad\cdot \frac{(q^{bD};q^{bD})_{\infty}^3\, \Theta(x_1^{c-b}y_1^{a-b} x_2y_2 q^{-D(u+v)};q^{bD})}{\Theta(x_1^cy_1^{-b}x_2q^{-Du};q^{bD})\Theta(x_1^{-b}y_1^a y_2 q^{-Dv};q^{bD})}.
\end{split}\]
\end{theorem}

In Section \ref{section:proof-main}, we prove Theorem \ref{theo:main} and give an example for $t=1$.  In Section \ref{section:proof-hecke} we prove Theorem \ref{theo:general-hecke}.  In Section \ref{section:proof-fabc}, we use Theorem \ref{theo:general-hecke} to expand the general form of a Hecke-type double-sum (\ref{equation:fabc-def}) in terms of theta functions and Appell functions, similar to \cite[Theorems 1.3, 1.4]{HM}.   In Section \ref{section:examples}, we give examples where we use Theorem \ref{theo:main} to express the dual functions $U_{t}^{(m)}(-x;q)$ for $t=2,3,$ in terms of Hecke-type double-sums (\ref{equation:fabc-def}), which can be evaluated in terms of theta functions and Appell functions by using Theorem \ref{theo:general-hecke}.   In Section \ref{section:proof-thetaless}, we prove a theta-less version of Theorem \ref{theo:general-hecke}.  In Section \ref{equation:f123-mxzq}, we derive an Appell function form for $U_{1}(-x;q)$.

\section{Proof of Theorem \ref{theo:main}}\label{section:proof-main}
First we consider the theta functions $\theta_{p,m}$ and $\theta_{p,m}^*$ as in Definition \ref{def:theta}.

\begin{lemma}\label{lem:theta}
We have
\begin{enumerate}
\item[\textnormal{(a)}] $\theta_{p,m}=q^{p+t}\, \theta_{p+2t,m}$ and $\theta_{p+2t,m}^* = \theta_{p,m}^*$;
\item[\textnormal{(b)}] $\theta_{p,m}=-q^{-m-t+\frac12}\, \theta_{p,m+2t-1}$ and $\theta_{p,m+2t-1}^* = -\theta_{p,m}^*$;
\item[\textnormal{(c)}] $\theta_{-p,m}= \theta_{p,-m}=-\theta_{p,m}$ and $\theta_{-p,m}^*=\theta_{p,-m}^*= -\theta_{p,m}^*$;
\item[\textnormal{(d)}] $\theta_{p,m}^*$ is a holomorphic modular form of weight 1.
\end{enumerate}
\end{lemma}

\begin{remark}\label{firstremark}
For $s\in\Z$ fixed we have
\[\underset{r\not\equiv s\smod{2}}{\sum_{r\in\Z}} (-1)^{\frac{r-s-1}2} q^{\frac18 r^2 + \frac{4t-1}4rs+\frac18 s^2 +\frac{p+m}2r+\frac{p-m}2 s}=0.\]
To see this replace $r$ by $-r-2(4t-1)s-4(p+m)$ in the summation.
In a similar way we have
\[\underset{s\not\equiv r\smod{2}}{\sum_{s\in\Z}} (-1)^{\frac{r-s-1}2} q^{\frac18 r^2 + \frac{4t-1}4rs+\frac18 s^2 +\frac{p+m}2r+\frac{p-m}2 s}=0\]
for fixed $r\in\Z$.
\end{remark}

\begin{proof}[Proof of Lemma \ref{lem:theta}]
(a) Using $\sg(r,s)=\sg(r-1,s-1)+\delta(r)+\delta(s)$, where
\begin{equation*}
\delta(r):=\begin{cases} 1 &\text{if}\ r=0,\\ 0 &\text{otherwise,}\end{cases}
\end{equation*}
and Remark \ref{firstremark} we find
\[\theta_{p,m}=\sum_{r\not\equiv s\smod{2}} \sg(r-1,s-1)\,(-1)^{\frac{r-s-1}2} q^{\frac18 r^2 + \frac{4t-1}4rs+\frac18 s^2 +\frac{p+m}2r+\frac{p-m}2 s}\]
and so replacing $(r,s)$ by $(r+1,s+1)$ we obtain
\[\theta_{p,m}= \sum_{r\not\equiv s\smod{2}} \sg(r,s)\,(-1)^{\frac{r-s-1}2} q^{\frac18 r^2 + \frac{4t-1}4rs+\frac18 s^2 +\frac{p+2t+m}2r+\frac{p+2t-m}2 s+p+t}=q^{p+t}\,\theta_{p+2t,m},\]
from which the second result immediately follows.\\
(b) In a similar way we get these identities by replacing $(r,s)$ by $(r-1,s+1)$.\\
(c) Replacing $(r,s)$ by $(-s,-r)$ and using
\[\sg(-s,-r)=\sg(-r,-s) =\sg(-r-1,-s-1)+\delta(-r)+\delta(-s)= -\sg(r,s)+\delta(r)+\delta(s)\]
and Remark \ref{firstremark} we get
\[\theta_{p,m}= -\sum_{r\not\equiv s\smod{2}} \sg(r,s)\,(-1)^{\frac{r-s-1}2} q^{\frac18 r^2 + \frac{4t-1}4rs+\frac18 s^2 +\frac{-p+m}2r+\frac{-p-m}2 s}=-\theta_{-p,m},\]
which also directly gives $\theta_{-p,m}^*= -\theta_{p,m}^*$.
Further, $\theta_{p,-m}=-\theta_{p,m}$ and $\theta_{p,-m}^*= -\theta_{p,m}^*$ follow directly by replacing $(r,s)$ by $(s,r)$.\\
(d) We have
\[ \begin{split}
\frac18 \Bigl( r+ \frac p{2t} -\frac m{2t-1} \Bigr)^2 &+\frac{4t-1}4 \Bigl( r+ \frac p{2t} -\frac m{2t-1} \Bigr)\Bigl( s+ \frac p{2t} +\frac m{2t-1} \Bigr) + \frac18 \Bigl( s+ \frac p{2t} +\frac m{2t-1} \Bigr)^2\\
&= \frac18r^2+ \frac{4t-1}4rs + \frac18 s^2 +\frac{p+m}2r+\frac{p-m}2s -\frac{m^2}{2(2t-1)}+\frac{p^2}{4t},
\end{split}\]
so using Remark \ref{firstremark} we can write $\theta_{p,m}^*$ as
\[ \theta_{p,m}^*=\underset{\tilde r-\tilde s +\frac{2m}{2t-1} \equiv 1 \smod{2}}{\sum_{\tilde r \in \frac p{2t}-\frac m{2t-1}+\Z,\, \tilde s \in \frac p{2t}+\frac m{2t-1}+\Z}} \frac12\bigl( \operatorname{sgn}(\tilde r)+\operatorname{sgn}(\tilde s)\bigr) \, (-1)^{\frac12(\tilde r-\tilde s+\frac{2m}{2t-1} -1)}\,q^{\frac18 \tilde r^2 +\frac{4t-1}4 \tilde r\tilde s +\frac18 \tilde s^2}.\]
From \cite{Zw} we get that this holomorphic theta function can be completed to a non-holomor\-phic theta function that transforms as a modular form of weight 1 on some subgroup of $\operatorname{SL}_2(\Z)$.
For this we have to replace $\operatorname{sgn}(\tilde r)$ by $E\bigl(2\sqrt{t(2t-1)}\,\tilde r y^{1/2}\bigr)$ and $\operatorname{sgn}(\tilde s)$ by $E\bigl(2\sqrt{t(2t-1)}\,\tilde s y^{1/2}\bigr)$, where $q=e^{2\pi i\tau}$, $y=\operatorname{Im}(\tau)>0$ and $E(z)=2\int_0^z e^{-\pi u^2} du$.
Now using $E(z)=\operatorname{sgn}(z)- \operatorname{sgn}(z) \beta(z^2)$ (Lemma 1.7 in \cite{Zw}) with $\beta(x) = \int_x^\infty u^{-1/2} e^{-\pi u} du$, and again Remark \ref{firstremark} we can easily see that the parts coming from $\beta$ vanish and so the holomorphic and the non-holomorphic version actually agree.
Hence $\theta_{p,m}^*$ (considered as a function of $\tau$) transforms as a holomorphic modular form of weight 1 on some subgroup.
\end{proof}

Next we establish a functional equation for both $g_{t,m}$ and $f_{t,m}$.

\begin{lemma}\label{lem:lem2}
We have
\[g_{t,m}(qx)= x^{2t}q^t g_{t,m}(x) +x^{t-m+1} q^{\frac 18 -\frac m2 +\frac t2} \frac{(q)_\infty^3}{\Theta(x;q)}(1-x^{2m}q^m)+ \sum_{k=0}^{2t-1} x^k q^{\frac k2} \theta_{k-t,m}\]
and
\[f_{t,m}(qx)=-x^{2t-1}q^t f_{t,m}(x)+x^{t-m}q^{-m+1}(1-x^{2m}q^m) +x^{-1} q^{-\frac m2 + \frac 78}\, \frac{\Theta(x;q)}{(q)_\infty^3} \sum_{k=0}^{2t-1} x^k q^{\frac{k-t}2} \theta_{k-t,m}.\]
\end{lemma}

\begin{proof}[Proof of Lemma \ref{lem:lem2}]
We consider the sum
\[x^{-2t}q^{-t} \sum_{r\not\equiv s\smod{2}} \sg(r,s) (-1)^{\frac{r-s-1}2} q^{\frac18 r^2 + \frac{4t-1}4rs+\frac18 s^2 +\frac{t+m}2r+\frac{t-m}2 s}\,\frac{1-x^{2t}q^{t(r+s+3)}}{1-xq^{\frac{r+s+3}2}}\]
and rewrite it in two different ways.
First using
\[ \frac{1-x^{2t}q^{t(r+s+3)}}{1-xq^{\frac{r+s+3}2}}= \sum_{k=0}^{2t-1} x^k q^{\frac{k(r+s+3)}2}\]
we get that it equals
\[\begin{split}
x^{-2t}q^{-t} &\sum_{k=0}^{2t-1} x^k q^{\frac{3k}2} \sum_{r\not\equiv s\smod{2}} \sg(r,s)\,(-1)^{\frac{r-s-1}2} q^{\frac18 r^2 + \frac{4t-1}4rs+\frac18 s^2 +\frac{k+t+m}2r+\frac{k+t-m}2 s}\\
&= \sum_{k=0}^{2t-1} x^{k-2t} q^{\frac{3k}2-t} \theta_{k+t,m}=\sum_{k=0}^{2t-1} x^{k-2t} q^{\frac k2-t} \theta_{k-t,m}.
\end{split}\]
On the other hand, we can split it as the sum of
\[x^{-2t}q^{-t}\sum_{r\not\equiv s\smod{2}} \sg(r,s)\, \frac{(-1)^{\frac{r-s-1}2} q^{\frac18 r^2 + \frac{4t-1}4rs+\frac18 s^2 +\frac{t+m}2r+\frac{t-m}2 s}}{1-xq^{\frac{r+s+3}2}}=x^{-2t}q^{-t}g_{t,m}(qx)\]
and
\[ -\sum_{r\not\equiv s\smod{2}} \sg(r,s) \frac{(-1)^{\frac{r-s-1}2} q^{\frac18 r^2 + \frac{4t-1}4rs+\frac18 s^2 +\frac{3t+m}2r+\frac{3t-m}2 s+2t}}{1-xq^{\frac{r+s+3}2}}.\]
Replacing $(r,s)$ by $(r-1,s-1)$ in this latter sum and using $\sg(r-1,s-1)=\sg(r,s) -\delta(r)-\delta(s)$ we find that it equals
\[\begin{split}
-&\sum_{r\not\equiv s\smod{2}} \sg(r-1,s-1)\, \frac{(-1)^{\frac{r-s-1}2} q^{\frac18 r^2 + \frac{4t-1}4rs+\frac18 s^2 +\frac{t+m}2r+\frac{t-m}2 s}}{1-xq^{\frac{r+s+1}2}}\\
&= -g_{t,m}(x)+\sum_{s\equiv 1\smod{2}} \frac{(-1)^{\frac{s+1}2} q^{\frac18 s^2 +\frac{t-m}2 s}}{1-xq^{\frac{s+1}2}}+\sum_{r\equiv 1\smod{2}}  \frac{(-1)^{\frac{r-1}2} q^{\frac18 r^2 +\frac{t+m}2 r}}{1-xq^{\frac{r+1}2}}\\
&= -g_{t,m}(x)+q^{\frac18 + \frac m2-\frac t2} \sum_{k\in\Z} \frac{(-1)^k q^{\frac 12 k^2+(t-m-\frac12)k}}{1-xq^k}-q^{\frac18-\frac m2-\frac t2}\sum_{k\in\Z} \frac{(-1)^k q^{\frac 12 k^2+(t+m-\frac12)k}}{1-xq^k}.
\end{split}\]
Using
\[\sum_{k\in\Z} \frac{(-1)^k q^{\frac 12 k^2+(n+\frac12)k}}{1-xq^k}=\frac{(q)_\infty^3}{x^n \Theta(x;q)},\]
we then obtain
\[\sum_{k=0}^{2t-1} x^{k-2t} q^{\frac k2-t} \theta_{k-t,m}=x^{-2t}q^{-t}g_{t,m}(qx)-g_{t,m}(x) -x^{-t-m+1}q^{\frac18-\frac m2 -\frac t2} \frac{(q)_\infty^3}{\Theta(x;q)}(1-x^{2m}q^m),\]
from which the functional equation for $g_{t,m}$ follows by multiplying with $x^{2t}q^t$.
Now using the well-known elliptic transformation property
\begin{equation}
    \Theta(q^nx;q)=(-1)^{n}q^{-\binom{n}{2}}x^{-n}\Theta(x;q),
    \label{equation:j-elliptic}
\end{equation}
we directly get the functional equation for $f_{t,m}$.
\end{proof}

\begin{example}
We consider the case $t=1$.
From Lemma \ref{lem:theta} we get $\theta_{-1,m}=\theta_{0,m}=0$, so Lemma \ref{lem:lem2} gives
\[f_{1,m}(qx)=-xq f_{1,m}(x)+x^{-m+1}q^{-m+1}-x^{m+1}q.\]
The function
\[ f(x)=(-1)^{m+1}q^{\frac12m^2-\frac12m}\sum_{n=-m+1}^m (-1)^n q^{-\frac12n^2+\frac12n}x^n\]
satisfies the same recursion.
Since both $f_{1,m}$ and $f$ have no poles, this then gives
\[ f_{1,m}(x)=f(x)= (-1)^{m+1}q^{\frac12m^2-\frac12m}\sum_{n=-m+1}^m (-1)^n q^{-\frac12n^2+\frac12n}x^n.\]
\end{example}

Now we prove Theorem \ref{theo:main} by solving the functional equation for $f_{t,m}$.

\begin{proof}[Proof of Theorem \ref{theo:main}]
Since $f_{t,m}$ has no poles, we can write it as a Laurent series:
\begin{equation}\label{eq:laurent}
f_{t,m}(x) = \sum_{r\in\Z} (-1)^r q^{-\frac1{2(2t-1)} r^2 -\frac1{2(2t-1)} r} a_r\, x^{-r}
\end{equation}
for all $x\in\C\smallsetminus\{0\}$.
Note that the extra factor in the coefficients is just for convenience.
Plugging this series in the functional equation for $f_{t,m}$ and comparing coefficients on both sides, we get that the coefficients $a_r$ satisfy the recurrence relation
\begin{equation}\label{eq:recur}
a_r= a_{r+2t-1} +b_r+ \begin{cases} C & \text{if}\ r=-t+m,\\ -C & \text{if}\ r=-t-m,\\ 0&\text{otherwise,}\end{cases}
\end{equation}
where
\begin{equation}\label{eq:bandC}
\begin{split}
b_r&:=- q^{-\frac t2 -\frac m2+\frac78 +\frac t{2t-1}r^2 +\frac t{2t-1}r}\frac1{(q)_\infty^3} \sum_{k=0}^{2t-1} (-1)^k q^{\frac12k^2+rk} \theta_{k-t,m},\\
C&:= (-1)^{t+m} q^{-\frac{(3t-m-2)(t+m-1)}{2(2t-1)}}.
\end{split}
\end{equation}
Further, from \eqref{eq:laurent} we also get
\[\lim_{r\rightarrow \pm \infty} a_r=0,\]
which together with \eqref{eq:recur} yields
\begin{equation}\label{eq:sum}
\sum_{n\equiv \ell \smod{2t-1}} b_n= \begin{cases} -C & \text{if}\ \ell \equiv -t+m\smod{2t-1},\\ C & \text{if}\ \ell\equiv -t-m\smod{2t-1},\\ 0&\text{otherwise.}\end{cases}
\end{equation}
Now consider
\[ \widetilde a_r := \sum_{l\in\Z} \sg(r,l)\, b_{r+l(2t-1)}.\]
We claim that $\widetilde a_r$ satisfies the same recurrence relation as $a_r$:
\[ \begin{split}
\widetilde a_r- \widetilde a_{r+2t-1} &= \sum_{l\in\Z} \bigl( \sg(r,l)-\sg(r+2t-1,l-1)\bigr)\, b_{r+l(2t-1)}\\
&=\sum_{l\in\Z} \bigl( \delta(l) -\delta(r+1)-\delta(r+2)-\ldots - \delta(r+2t-1)\bigr)\, b_{r+l(2t-1)}\\
&= b_r- \bigl(\delta(r+1)+\delta(r+2)+\ldots + \delta(r+2t-1)\bigr) \cdot \sum_{n\equiv r \smod{2t-1}} b_n\\
&= b_r+\begin{cases} C & \text{if}\ r=-t+m,\\ -C & \text{if}\ r=-t-m,\\ 0&\text{otherwise,}\end{cases}
\end{split}\]
where in the last step we have used \eqref{eq:sum} and that $1\leq m<t$.
Further, we can easily see that $\lim_{r\rightarrow \infty} \widetilde a_r=0$.
Hence we conclude
\begin{equation}\label{eq:sol}
a_r =\widetilde a_r = \sum_{l\in\Z} \sg(r,l)\, b_{r+l(2t-1)}
\end{equation}
(the difference $a_r-\widetilde a_r$ is $2t-1$ periodic and goes to 0 for $r\rightarrow \infty$, so it is identically 0).
Now combining \eqref{eq:laurent}  with \eqref{eq:sol} and \eqref{eq:bandC} we get
\[\begin{split}
&f_{t,m}(x)= \sum_{r,l\in\Z} \sg(r,l)\,(-1)^r q^{-\frac1{2(2t-1)} r^2 -\frac1{2(2t-1)} r} x^{-r} b_{r+l(2t-1)}\\
&= -\frac{q^{-\frac t2-\frac m2+\frac78}}{(q)_\infty^3}\sum_{k=0}^{2t-1} (-1)^k q^{\frac12 k^2} \theta_{k-t,m} \sum_{r,l\in\Z} \sg(r,l)\, (-1)^r q^{\frac12 r^2+2trl+t(2t-1)l^2+(k+\frac12)r + (2tk-k+t)l} x^{-r}\\
&= -\frac{q^{-\frac34t-\frac m2+\frac78}}{(q)_\infty^3}\sum_{k=0}^{2t-1} (-1)^k q^{\frac1{4t} (k-t)^2} \theta_{k-t,m} \sum_{r,l\in\Z} \sg(r,l)\, (-1)^r q^{\frac12 r^2+(2tl+k)r+\frac{2t-1}{4t}(2tl+k)^2+\frac12 r + \frac12(2tl+k)} x^{-r}\\
&= -\frac{q^{-\frac34t-\frac m2+\frac78}}{(q)_\infty^3}\sum_{k=0}^{2t-1} (-1)^k q^{\frac1{4t} (k-t)^2} \theta_{k-t,m} \underset{l\equiv k\smod{2t}}{\sum_{r,l\in\Z}} \sg(r,l)\, (-1)^r q^{\frac12 r^2+lr+\frac{2t-1}{4t}l^2+\frac12 r + \frac12l} x^{-r}.
\end{split}\]
Since both
\[ (-1)^k q^{\frac1{4t} (k-t)^2} \theta_{k-t,m}\qquad \text{and} \qquad \underset{l\equiv k\smod{2t}}{\sum_{r,l\in\Z}} \sg(r,l)\, (-1)^r q^{\frac12 r^2+lr+\frac{2t-1}{4t}l^2+\frac12 r + \frac12l} x^{-r}\]
depend only on $k$ modulo $2t$, we can replace $\sum_{k=0}^{2t-1}$ by $\sum_{k\operatorname{mod} 2t}$ to get the desired result.
\end{proof}

Using equation \eqref{eq:sum} we can actually obtain a useful identity involving the theta functions $\theta_{k,m}^*$.

\begin{theorem}\label{theo:twee}
For $t\geq 2$, $1\leq m <t$ and $\ell\in\Z$ we have
\[ \sum_{k\operatorname{mod}2t} (-1)^k \theta_{k,m}^* \underset{r\equiv 2t\ell+(2t-1)k \smod{2t(2t-1)}}{\sum_{r\in\Z}} q^{\frac{1}{4t(2t-1)}r^2}=\begin{cases} (-1)^m q^{\frac18}(q)_\infty^3 & \text{if}\ \ell \equiv m\smod{2t-1},\\ (-1)^{m+1}q^{\frac18}(q)_\infty^3 & \text{if}\ \ell\equiv -m\smod{2t-1},\\ 0&\text{otherwise.}\end{cases}\]
\end{theorem}

\begin{proof}[Proof of Theorem \ref{theo:twee}]
From \eqref{eq:bandC} we get
\[ \begin{split}
&\sum_{n\equiv \ell\smod{2t-1}} b_n = \sum_{r\in\Z} b_{(2t-1)r+\ell}\\
&\qquad = -\frac{q^{-\frac t2-\frac m2+\frac78}}{(q)_\infty^3} \sum_{r\in\Z} q^{\frac t{2t-1}((2t-1)r+\ell)^2 +\frac t{2t-1}((2t-1)r+\ell)} \sum_{k=0}^{2t-1} (-1)^k q^{\frac12k^2+((2t-1)r+\ell)k} \theta_{k-t,m}\\
&\qquad = -\frac{q^{-\frac34t-\frac m2+\frac78-\frac{t}{4(2t-1)}}}{(q)_\infty^3} \sum_{k=0}^{2t-1} (-1)^k q^{\frac1{4t}(k-t)^2} \theta_{k-t,m} \sum_{r\in\Z} q^{\frac{1}{4t(2t-1)}(2t(2t-1)r+2t\ell+(2t-1)k+t)^2}\\
&\qquad = -\frac{q^{-\frac34t-\frac m2+\frac78-\frac{t}{4(2t-1)}}}{(q)_\infty^3} \sum_{k\operatorname{mod}2t} (-1)^k q^{\frac1{4t}(k-t)^2} \theta_{k-t,m} \underset{r\equiv 2t\ell+(2t-1)k+t \smod{2t(2t-1)}}{\sum_{r\in\Z}} q^{\frac{1}{4t(2t-1)}r^2},
\end{split}\]
so \eqref{eq:sum} gives
\[\begin{split}
\sum_{k\operatorname{mod}2t} (-1)^k \theta_{k-t,m}^* &\underset{r\equiv 2t\ell+(2t-1)k+t \smod{2t(2t-1)}}{\sum_{r\in\Z}} q^{\frac{1}{4t(2t-1)}r^2}\\
&=\begin{cases} (-1)^{t+m} q^{\frac18}(q)_\infty^3 & \text{if}\ \ell \equiv -t+m\smod{2t-1},\\ (-1)^{t+m+1} q^{\frac18}(q)_\infty^3 & \text{if}\ \ell\equiv -t-m\smod{2t-1},\\ 0&\text{otherwise.}\end{cases}
\end{split}\]
Now replacing $(k,\ell)$ by $(k+t,\ell -t)$ and multiplying with $(-1)^t$ we get the desired result.
\end{proof}

\section{Proof of Theorem \ref{theo:general-hecke}}\label{section:proof-hecke}

First we assume that $x_1$, $x_2$, $y_1$, $y_2$ satisfy
\[ |q|^{D} < |x_1^cy_1^{-b}x_2|< 1\qquad \text{and}\qquad |q|^{D} < |x_1^{-b}y_1^a y_2|< 1.\]
We consider the indefinite theta series
\[ \sum_{r_1,s_1,r_2,s_2\in\Z} \sg(r_1,s_1)\, q^{\frac12a r_1^2+br_1s_1+\frac12cs_1^2+\frac12 cDr_2^2+\frac12a Ds_2^2}\, x_1^{r_1}y_1^{s_1} x_2^{r_2} y_2^{s_2},\]
which immediately splits as a product:
\[ \sum_{r_2\in\Z} q^{\frac12 cDr_2^2}\, x_2^{r_2} \sum_{s_2\in\Z} q^{\frac12 aDs_2^2}\, y_2^{s_2} \sum_{r_1,s_1\in\Z} \sg(r_1,s_1)\, q^{\frac12a r_1^2+br_1s_1+\frac12cs_1^2}\, x_1^{r_1}y_1^{s_1}.\]
Note that the last double sum is the Hecke-type double sum that we're actually interested in.
On the other hand, we can split $\sg(r_1,s_1)$ as
\[ \sg(r_1,s_1)= \sg(r_1,cr_2-r_1) +\sg(s_1,as_2-s_1) - \sg(cr_2-r_1,as_2-s_1)\]
and use this to split the series into the sum of three parts.
We consider each of these parts separately and start with the part coming from $\sg(r_1,cr_2-r_1)$:
\[ \begin{split}
& \sum_{r_1,s_1,r_2,s_2\in\Z} \sg(r_1,cr_2-r_1)\, q^{\frac12a r_1^2+br_1s_1+\frac12cs_1^2+\frac12 cDr_2^2+\frac12a Ds_2^2}\, x_1^{r_1}y_1^{s_1} x_2^{r_2} y_2^{s_2}\\
&\qquad= \sum_{s_2\in\Z} q^{\frac12 aDs_2^2}\, y_2^{s_2} \sum_{r_1,s_1,r_2\in\Z} \sg(r_1,cr_2-r_1)\, q^{\frac12a r_1^2+br_1s_1+\frac12cs_1^2+\frac12 cDr_2^2}\, x_1^{r_1}y_1^{s_1} x_2^{r_2}
\end{split}\]
In the triple sum we write $r_1$ as $ck+t$ with $k\in\Z$ and $0\leq t < c$, and replace $(s_1,r_2)$ by $(s_1-bk,r_2+k)$ to get
\[\begin{split}
\sum_{s_2\in\Z} q^{\frac12 aDs_2^2}\, y_2^{s_2}&\sum_{t=0}^{c-1} \sum_{s_1\in\Z} q^{\frac12a t^2+bts_1+\frac12cs_1^2}x_1^t y_1^{s_1}\\
&\cdot \sum_{k,r_2\in\Z} \sg(ck+t,cr_2-t)\, q^{\frac12 cDr_2^2} x_2^{r_2} \bigl(x_1^cy_1^{-b}x_2 q^{D(cr_2-t)}\bigr)^k.    
\end{split}\]
In the last double sum we replace $\sg(ck+t,cr_2-t)$ by $\sg(k,cr_2-t)$ and use
\begin{equation}\label{eq:geom}
\sum_{k\in\Z} \sg(k,n)\, z^k q^{nk} =\frac1{1-zq^n},\qquad |q| < |z| < 1
\end{equation}
with $q$ replaced by $q^{D}$, $n=cr_2-t$ and $z=x_1^cy_1^{-b}x_2$ to get
\[\sum_{s_2\in\Z} q^{\frac12 aDs_2^2}\, y_2^{s_2}\sum_{t=0}^{c-1} \sum_{s_1\in\Z} q^{\frac12a t^2+bts_1+\frac12cs_1^2}x_1^t y_1^{s_1}\sum_{r_2\in\Z}\frac{q^{\frac12 cDr_2^2} x_2^{r_2}}{1-x_1^cy_1^{-b}x_2 q^{D(cr_2-t)}}.\]
By symmetry we get that the part coming from $\sg(s_1,as_2-s_1)$ equals
\[ \begin{split}
    \sum_{r_1,s_1,r_2,s_2\in\Z} &\sg(s_1,as_2-s_1)\, q^{\frac12a r_1^2+br_1s_1+\frac12cs_1^2+\frac12 cDr_2^2+\frac12a Ds_2^2}\, x_1^{r_1}y_1^{s_1} x_2^{r_2} y_2^{s_2}\\
    &=\sum_{r_2\in\Z} q^{\frac12 cDr_2^2}\, x_2^{r_2}\sum_{t=0}^{a-1} \sum_{r_1\in\Z} q^{\frac12a r_1^2+br_1t+\frac12ct^2}x_1^{r_1} y_1^t\sum_{s_2\in\Z}\frac{q^{\frac12 aDs_2^2} y_2^{s_2}}{1-x_1^{-b}y_1^a y_2 q^{D(as_2-t)}}.
\end{split}\]
Finally, we consider the part coming from $\sg(cr_2-r_1,as_2-s_1)$:
here we substitute $(cr_2-r_1,as_2-s_1)$ for $(r_1,s_1)$ and rewrite the exponent of $q$ to get
\[ \begin{split}
    &\sum_{r_1,s_1,r_2,s_2\in\Z} \sg(cr_2-r_1,as_2-s_1)\, q^{\frac12a r_1^2+br_1s_1+\frac12cs_1^2+\frac12 cDr_2^2+\frac12a Ds_2^2}\, x_1^{r_1}y_1^{s_1} x_2^{r_2} y_2^{s_2}\\
    &=\sum_{r_1,s_1,r_2,s_2\in\Z} \sg(r_1,s_1)\, q^{\frac12c(br_2-s_1)^2+\frac{ac}b(br_2-s_1)(bs_2-r_1) +\frac12a(bs_2-r_1)^2 +\frac Dbr_1s_1}\, x_1^{cr_2-r_1}y_1^{as_2-s_1} x_2^{r_2} y_2^{s_2}.
\end{split}\]
We now write $r_1$ as $bk-u$ where $u$ is in a complete residue system modulo $b$ and $k\in\Z$.
Similarly we write $s_1$ as $bl-v$.
Further, we substitute $(r_2+l,s_2+k)$ for $(r_2,s_2)$ to get
\[\begin{split}
\sum_{u,v\smod b}q^{\frac12au^2+buv+\frac12cv^2} x_1^{u}y_1^{v} &\sum_{r_2,s_2\in\Z} q^{\frac12 b^2cr_2^2+abcr_2s_2+\frac12ab^2s_2^2 +c(au+bv)r_2+a(bu+cv)s_2}(x_1^cx_2)^{r_2} (y_1^ay_2)^{s_2}\\
& \cdot \sum_{k,l\in\Z} \sg(bk-u,bl-v)\, q^{bDkl} \bigl(x_1^{-b}y_1^a y_2 q^{-Dv}\bigr)^k \bigl( x_1^c y_1^{-b}x_2 q^{-Du}\bigr)^l.
\end{split}\]
We know that
\[ \sum_{k,l\in\Z} \sg(k+\gamma_1,l+\gamma_2)\, q^{kl} z_2^k z_1^l = \frac{(q;q)_{\infty}^3 \Theta(z_1z_2;q)}{\Theta(z_1;q)\Theta(z_2;q)}\]
holds with the conditions $|q|^{\gamma_i+1}< |z_i| < |q|^{\gamma_i}$, where $\gamma_1,\gamma_2\in\Z$.
We replace $q$ by $q^{bD}$ and set $z_1=x_1^cy_1^{-b}x_2q^{-Du}$, $z_2=x_1^{-b}y_1^a y_2 q^{-Dv}$.
Then the conditions are indeed satisfied with $\gamma_1=[-u/b]$, $\gamma_2=[-v/b]$.
Hence we get
\begin{align*} \sum_{k,l\in\Z} \sg(k+\gamma_1,l+\gamma_2)\, q^{bDkl}& \bigl(x_1^{-b}y_1^a y_2 q^{-Dv}\bigr)^k \bigl( x_1^c y_1^{-b}x_2 q^{-Du}\bigr)^l\\
& = \frac{(q^{bD};q^{bD})_{\infty}^3\, \Theta(x_1^{c-b}y_1^{a-b} x_2y_2 q^{-D(u+v)};q^{bD})}{\Theta(x_1^cy_1^{-b}x_2q^{-Du};q^{bD})\Theta(x_1^{-b}y_1^a y_2 q^{-Dv};q^{bD})}.
\end{align*}
Since
\[ \sg(bk-u,bl-v)= \sg(k+\gamma_1,l+\gamma_2)\]
we thus obtain that the part coming from $\sg(cr_2-r_1,as_2-s_1)$ equals
\[\begin{split}
\sum_{u,v\smod b} q^{\frac12au^2+buv+\frac12cv^2} x_1^{u}y_1^{v} &\sum_{r_2,s_2\in\Z} q^{\frac12 b^2cr_2^2+abcr_2s_2+\frac12ab^2s_2^2 +c(au+bv)r_2+a(bu+cv)s_2}(x_1^cx_2)^{r_2} (y_1^ay_2)^{s_2}\\
& \cdot \frac{(q^{bD};q^{bD})_{\infty}^3\, \Theta(x_1^{c-b}y_1^{a-b} x_2y_2 q^{-D(u+v)};q^{bD})}{\Theta(x_1^cy_1^{-b}x_2q^{-Du};q^{bD})\Theta(x_1^{-b}y_1^a y_2 q^{-Dv};q^{bD})}.
\end{split}\]
Now combining the different parts we obtain
\[\begin{split}
 \sum_{r_2\in\Z} q^{\frac12 cDr_2^2}\, & x_2^{r_2} \sum_{s_2\in\Z} q^{\frac12 aDs_2^2}\, y_2^{s_2} \sum_{r_1,s_1\in\Z} \sg(r_1,s_1)\, q^{\frac12a r_1^2+br_1s_1+\frac12cs_1^2}\, x_1^{r_1}y_1^{s_1}\\
 &=\sum_{s_2\in\Z} q^{\frac12 aDs_2^2}\, y_2^{s_2}\sum_{t=0}^{c-1} \sum_{s_1\in\Z} q^{\frac12a t^2+bts_1+\frac12cs_1^2}x_1^t y_1^{s_1}\sum_{r_2\in\Z}\frac{q^{\frac12 cDr_2^2} x_2^{r_2}}{1-x_1^cy_1^{-b}x_2 q^{D(cr_2-t)}}\\
 &\quad +\sum_{r_2\in\Z} q^{\frac12 cDr_2^2}\, x_2^{r_2}\sum_{t=0}^{a-1} \sum_{r_1\in\Z} q^{\frac12a r_1^2+br_1t+\frac12ct^2}x_1^{r_1} y_1^t\sum_{s_2\in\Z}\frac{q^{\frac12 aDs_2^2} y_2^{s_2}}{1-x_1^{-b}y_1^a y_2 q^{D(as_2-t)}}\\
 &\qquad -\sum_{u,v\smod b} q^{\frac12au^2+buv+\frac12cv^2} x_1^{u}y_1^{v}\\
 &\qquad\qquad \cdot \sum_{r_2,s_2\in\Z} q^{\frac12 b^2cr_2^2+abcr_2s_2+\frac12ab^2s_2^2 +c(au+bv)r_2+a(bu+cv)s_2}(x_1^cx_2)^{r_2} (y_1^ay_2)^{s_2}\\
&\qquad\qquad \cdot \frac{(q^{bD};q^{bD})_{\infty}^3\, \Theta(x_1^{c-b}y_1^{a-b} x_2y_2 q^{-D(u+v)};q^{bD})}{\Theta(x_1^cy_1^{-b}x_2q^{-Du};q^{bD})\Theta(x_1^{-b}y_1^a y_2 q^{-Dv};q^{bD})}
 \end{split}\]
for $|q|^{D} < |x_1^cy_1^{-b}x_2|< 1$ and $|q|^{D} < |x_1^{-b}y_1^a y_2|< 1$.
We fix $x_1$ and $y_1$ and consider both sides of this equation as a function of $x_2$ and $y_2$.
The left hand side is holomorphic in both variables and the right hand side meromorphic, with possible poles when $x_1^cy_1^{-b}x_2=q^{Dm}$ or $x_1^{-b}y_1^a y_2=q^{Dn}$ for some $m,n\in\Z$.
Using the identity theorem we thus get that the equation holds for all $x_2$ and $y_2$ outside of poles.
Dividing by
\[ \sum_{r_2\in\Z} q^{\frac12 cDr_2^2}\, x_2^{r_2} \sum_{s_2\in\Z} q^{\frac12 aDs_2^2}\, y_2^{s_2}\]
and rewriting in terms of $\Theta$ and $m$ we obtain the desired result.

\section{A corollary to Theorem \ref{theo:general-hecke}}\label{section:proof-fabc}

This leads us to the most general identity for Hecke-type double-sums.  To state our results, we define the following expression involving Appell functions:
\begin{definition} Let $a,b,$ and $c$ be positive integers with  $D:=b^2-ac>0$.  Then
\begin{align*} 
g_{a,b,c}(x,y,z_1,z_0;q)
&:=\sum_{t=0}^{a-1}(-y)^tq^{c\binom{t}{2}}\Theta(q^{bt}x;q^a)m\Big (-q^{a\binom{b+1}{2}-c\binom{a+1}{2}-tD}\frac{(-y)^a}{(-x)^b},z_0;q^{aD}\Big ) \\
&\ \ \ \ \ +\sum_{t=0}^{c-1}(-x)^tq^{a\binom{t}{2}}\Theta(q^{bt}y;q^c)m\Big (-q^{c\binom{b+1}{2}-a\binom{c+1}{2}-tD}\frac{(-x)^c}{(-y)^b},z_1;q^{cD}\Big ).\notag
\end{align*}
\end{definition}
\begin{theorem}\label{theo:general-fabc}Let $a,b,$ and $c$ be positive integers with  $D:=b^2-ac>0$. For generic $x$ and $y$, we have
\begin{align*}
& f_{a,b,c}(x,y;q)=g_{a,b,c}(x,y,-1,-1;q)+\frac{1}{\Theta(-1;q^{aD})\Theta(-1;q^{cD})}\cdot \theta_{a,b,c}(x,y;q),
\end{align*}
where
\begin{align*}
&\theta_{a,b,c}(x,y;q):=
\sum_{d^*=0}^{b-1}\sum_{e^*=0}^{b-1}q^{a\binom{d-c/2}{2}+b( d-c/2 ) (e+a/2  )+c\binom{e+a/2}{2}}(-x)^{d-c/2}(-y)^{e+a/2}\\
&\cdot\sum_{f=0}^{b-1}q^{ab^2\binom{f}{2}+\big (a(bd+b^2+ce)-ac(b+1)/2 \big )f} (-y)^{af}
\cdot \Theta(-q^{c\big ( ad+be+a(b-1)/2+abf \big )}(-x)^{c};q^{cb^2})\\
&\cdot \Theta(-q^{a\big ( (d+b(b+1)/2+bf)(b^2-ac) +c(a-b)/2\big )}(-x)^{-ac}(-y)^{ab};q^{ab^2D})\\
&\cdot \frac{(q^{bD};q^{bD})_{\infty}^3\Theta(q^{ D(d+e)+ac-b(a+c)/2}(-x)^{b-c}(-y)^{b-a};q^{bD})}
{\Theta(q^{De+a(c-b)/2}(-x)^b(-y)^{-a};q^{bD})\Theta(q^{Dd+c(a-b)/2}(-y)^b(-x)^{-c};q^{bD})}.
\end{align*}
Here $d:=d^*+\{c/2 \}$ and $e:=e^*+\{ a/2\}$, with  $0\le \{\alpha \}<1$ denoting fractional part of $\alpha$.
\end{theorem}

\begin{proof}[Proof of Theorem \ref{theo:general-fabc}]
We recall Theorem \ref{theo:general-hecke}.  Setting $x_1=-xq^{-\frac12a}$ and $y_1=-yq^{-\frac12c}$ the left hand side of this equation equals $f_{a,b,c}(x,y;q)$.
The first two terms on the right hand side exactly give
\[ g_{a,b,c}(x,y,-x_2q^{\frac12cD},-y_2q^{\frac12aD};q),\]
so we also have to set
\[ x_2 = q^{-\frac12cD}\qquad\text{and}\qquad y_2= q^{-\frac12aD}.\]
With these values
\[ \sum_{r_2\in\Z} q^{\frac12 cDr_2^2}\, x_2^{r_2} \sum_{s_2\in\Z} q^{\frac12 aDs_2^2}\, y_2^{s_2}\]
matches up with
\[ \Theta(-1;q^{aD})\Theta(-1;q^{cD}).\]
What remains to be checked is that
\begin{align}
\theta_{a,b,c}(x,y;q)=-&\sum_{u,v \pmod {b}} q^{\frac{1}{2}au^2+buv+\frac{1}{2}cv^2}x_{1}^{u}y_{1}^{v}\label{equation:theta-abc-v1}\\
&\cdot \sum_{r_2,s_2 \in \mathbb{Z}}q^{\frac{1}{2}b^2cr_2^2+abcr_2s_2+\frac{1}{2}ab^2s_2^2
+c(au+bv)r_2+a(bu+cv)s_2}(x_1^cx_2)^{r_2}(y_{1}^ay_2)^{s_2}\notag\\
& \cdot \frac{(q^{bD};q^{bD})_{\infty}^3\Theta(x_1^{c-b}y_{1}^{a-b}x_2y_2q^{-D(u+v)};q^{bD})}
{\Theta(x_{1}^cy_1^{-b}x_2q^{-Du};q^{bD})\Theta(x_{1}^{-b}y_1^{a}y_2q^{-Dv};q^{bD})}.
\notag
\end{align}
We make the substitutions $x_1=-xq^{-\frac{1}{2}a}$, $y_1=-yq^{-\frac{1}{2}c}$, $x_2=q^{-\frac{1}{2}cD}$, $y_2=q^{-\frac{1}{2}aD}$ and use (\ref{equation:j-elliptic}).  Omitting the lead minus sign, the right hand side of (\ref{equation:theta-abc-v1}) becomes
\begin{align}
&\sum_{u,v \pmod {b}} (-x)^{u}(-y)^{v}q^{a\binom{u}{2}+buv+c\binom{v}{2}}
\label{equation:theta-abc-v2}\\
&\cdot \sum_{r_2,s_2 \in \mathbb{Z}}
(-x)^{cr_2}(-y)^{as_2}q^{\frac{1}{2}b^2cr_2^2+abcr_2s_2+\frac{1}{2}ab^2s_2^2
+c(au+bv)r_2+a(bu+cv)s_2-\frac{1}{2}ac(r_2+s_2)-\frac{1}{2}(b^2-ac)(cr_2+as_2)}\notag\\
& \cdot \frac{(q^{bD};q^{bD})_{\infty}^3\Theta(q^{D(u+v+\frac{1}{2}a+\frac{1}{2}c)+ac-\frac{1}{2}b(a+c)}
(-x)^{b-c}(-y)^{b-a};q^{bD})}
{\Theta(q^{D(u+\frac{1}{2}c)+\frac{1}{2}c(a-b)}(-x)^{-c}(-y)^{b};q^{bD})\Theta(q^{D(v+\frac{1}{2}a)+\frac{1}{2}a(c-b)}(-x)^{b}(-y)^{-a};q^{bD})}.\notag
\end{align}
Let us focus on the double-sum over $r_2,s_2$ in (\ref{equation:theta-abc-v2}).  We replace $(r_2,s_2)$ with $(-ar+s, br+f)$ where $0\le f\le b-1$ to obtain
 {\allowdisplaybreaks\begin{align*}
\sum_{r_2,s_2 \in \mathbb{Z}}
&(-1)^{cr_2+as_2}q^{\frac{1}{2}b^2cr_2^2+abcr_2s_2+\frac{1}{2}ab^2s_2^2
+c(au+bv)r_2+a(bu+cv)s_2-\frac{1}{2}ac(r_2+s_2)-\frac{1}{2}(b^2-ac)(cr_2+as_2)}x^{cr_2}y^{as_2}\\
&=\sum_{f=0}^{b-1}(-y)^{af}q^{ab^2\binom{f}{2}+af\big (bu+cv+\frac{1}{2}c(a-1)\big )}\\
& \ \ \ \ \ \cdot \sum_{r\in\mathbb{Z}} \big ( (-x)^{-ac}(-y)^{ab}\big )^{r}
q^{ra\big ( (b^2-ac)(\binom{b}{2}+u+\frac{1}{2}c+bf)+\frac{1}{2}c(a-b)\big )}q^{ab^2(b^2-ac)\binom{r}{2}}\\
& \ \ \ \ \ \cdot \sum_{s\in\mathbb{Z}} (-x)^{cs}q^{sc\big ( \frac{1}{2}a(c-1)+abf+au+bv\big )}q^{b^2c\binom{s}{2}}\\
&=\sum_{f=0}^{b-1}(-y)^{af}q^{ab^2\binom{f}{2}+af\big (bu+cv+\frac{1}{2}c(a-1)\big )}\cdot  \Theta(-q^{c\big ( \frac{1}{2}a(c-1)+abf+au+bv\big )}(-x)^{c};q^{b^2c})\\
& \ \ \ \ \ \cdot \Theta(- q^{a\big ( D(\binom{b}{2}+u+\frac{1}{2}c+bf)+\frac{1}{2}c(a-b)\big )}
(-x)^{-ac}(-y)^{ab} ;q^{ab^2D}).
 \end{align*}}%
Collecting terms, (\ref{equation:theta-abc-v2}) becomes
\begin{align}
&\sum_{u,v \pmod {b}} (-x)^{u}(-y)^{v}q^{a\binom{u}{2}+buv+c\binom{v}{2}}
\label{equation:theta-abc-v3}\\
&\ \ \ \ \ \cdot \sum_{f=0}^{b-1}(-y)^{af}q^{ab^2\binom{f}{2}+af\big (bu+cv+\frac{1}{2}c(a-1)\big )}\cdot  \Theta(-q^{c\big ( \frac{1}{2}a(c-1)+abf+au+bv\big )}(-x)^{c};q^{b^2c})\notag\\
& \ \ \ \ \ \cdot \Theta(- q^{a\big ( D(\binom{b}{2}+u+\frac{1}{2}c+bf)+\frac{1}{2}c(a-b)\big )}
(-x)^{-ac}(-y)^{ab} ;q^{ab^2D})\notag\\
&\ \ \ \ \  \cdot \frac{(q^{bD};q^{bD})_{\infty}^3\Theta(q^{D(u+v+\frac{1}{2}a+\frac{1}{2}c)+ac-\frac{1}{2}b(a+c)}
(-x)^{b-c}(-y)^{b-a};q^{bD})}
{\Theta(q^{D(u+\frac{1}{2}c)+\frac{1}{2}c(a-b)}(-x)^{-c}(-y)^{b};q^{bD})\Theta(q^{D(v+\frac{1}{2}a)+\frac{1}{2}a(c-b)}(-x)^{b}(-y)^{-a};q^{bD})}.\notag
\end{align}

For the remainder of the proof, we regroups terms according to $(u+c/2)$ and $(v+a/2)$.  The two outermost sums are over $u$ and $v$ modulo $b$, so we can shift them by integers.  If $a$ or $c$ is not divisible by two, then there is a remaining $1/2$, so this is where the fractional parts come in.  Thinking modulo $b$, we let $d^*$ replace $u+\lfloor c/2\rfloor$ and let $e^*$ replace $v+\lfloor a/2 \rfloor$, where $\lfloor .. \rfloor$ is the greatest integer function.  Then $d:=d^* +\{c/2\}$, $e:=e^*+\{a/2\}$ where $0\le \{\alpha\}<1$ denotes the fractional part of $\alpha$.  Once everything is rewritten, one now has Theorem \ref{theo:general-fabc} but where we always have an $f-1$ instead of an $f$.  The sum over $f$ is modulo $b$ as well, so we replace $f$ with $f+1$ to have Theorem \ref{theo:general-fabc}.  On the other hand, one could just as easily stop with the above expression (\ref{equation:theta-abc-v3}).
\end{proof}

\section{Applications of the results}\label{section:examples}

Our first example needs no additional notation or results.

\begin{example}
We consider the case $t=2$ and $m=1$.
From Lemma \ref{lem:theta} we immediately get $\theta_{-2,1}=\theta_{0,1}=0$ and $\theta_{-1,1} =-\theta_{1,1}$, so Theorem \ref{theo:main} gives
\[ \begin{split}
f_{2,1}(x) &= - \frac{q^{-1} \theta_{1,1}}{(q)_\infty^3}\Bigl( \underset{l\equiv1 \smod{4}}{\sum_{r,l\in\Z}} - \underset{l\equiv3 \smod{4}}{\sum_{r,l\in\Z}} \Bigr) \sg(r,l)\, (-1)^r q^{\frac12 r^2+lr+\frac38l^2+\frac12 r + \frac12l} x^{-r}\\
&=- \frac{q^{-\frac18} \theta_{1,1}}{(q)_\infty^3}\sum_{r,l\in\Z} \sg(r,l)\, (-1)^{r+l} q^{\frac12r^2+2rl+\frac32 l^2+\frac32r+\frac52 l} x^{-r}.
\end{split}\]
Further, Theorem \ref{theo:twee} with $\ell=1$ gives
\[ \theta_{1,1}^*\, \Bigl(-\underset{r\equiv 7\smod{12}}{\sum_{r\in\Z}} + \underset{r\equiv 1\smod{12}}{\sum_{r\in\Z}}\Bigr)\, q^{\frac1{24}r^2}=-q^{\frac18}(q)_\infty^3,\]
so we find
\[ \theta_{1,1}^* = -q^{\frac1{12}} (q)_\infty^2 =-\eta^2\qquad \text{and}\qquad \theta_{1,1}= -q^{\frac18} (q)_\infty^2.\]
Thus we obtain
\begin{equation} f_{2,1}(x) = \frac1{(q)_\infty}\sum_{r,l\in\Z} \sg(r,l)\, (-1)^{r+l} q^{\frac12r^2+2rl+\frac32 l^2+\frac32r+\frac52 l} x^{-r},
\label{equation:f123-example}
\end{equation}
as in equation (4.1) of \cite{HL}.
\end{example}

Identity (\ref{equation:f123-example}) can thus be written
\begin{equation} f_{2,1}(x) = \frac1{(q)_\infty}f_{1,2,3}(x^{-1}q^2,q^4;q).\label{equation:f123-final}
\end{equation}
Unfortunately, \cite[Theorems 1.3, 1.4]{HM} do not evaluate sums of the form (\ref{equation:f123-final}).  However, from \cite[(4.7), (4.9)]{M1} we have the slightly rewritten
{\allowdisplaybreaks \begin{align}
f_{2,1}(x;q)&=(1-x )U_{1}^{(1)}(-x;q)\label{equation:f123-mxzq}\\
&=q^{-2}x^{-2}m(q^{-1}x^{-3},q^2x^2;q^3)-q^{-1}m(qx^{-3},q^2x^2;q^3)
-\frac{x}{q}\frac{\Theta(x;q)}{(q)_{\infty}}m(x^{2},x^{-1};q).\notag
\end{align}}%
We derive (\ref{equation:f123-mxzq}) in Section \ref{section:derive-f123}.

Let us go back to Theorem \ref{theo:main}.  We state a lemma that converts $\vartheta_{p,m}$ into double-sums, but we leave the proof until the end of the section.

\begin{lemma} \label{lem:thetaToHecke}We have
\begin{equation*}
\vartheta_{p,m}
=-q^{\frac{1}{8}+\frac{p-m}{2}}(f_{1,4t-1,1}(q^{p+m+2t},q^{p-m+1};q)
-q^mf_{1,4t-1,1}(q^{p+m+1},q^{p-m+2t};q)).
\end{equation*}
\end{lemma}

As was already pointed out, up to multiplication by $q$ to an appropriate exponent, these sums are modular \cite{Zw} and can be evaluated specifically in terms of theta functions \cite[Theorem 1.3]{HM}.  Rewriting Theorem \ref{theo:main} in terms of double-sums, we have
\begin{corollary}\label{cor:main} We have
\begin{align*}
f_{t,m}(x)&=\frac{q^{-m+1-t}}{J_{1}^3}
\sum_{k=0}^{2t-1}(-1)^{k}q^{\binom{k+1}{2}}\\
&\ \ \ \ \ \cdot \big (f_{1,4t-1,1}(q^{k+m+t},q^{k-t-m+1};q)
-q^mf_{1,4t-1,1}(q^{k-t+m+1},q^{k-m+t};q)\big )\\
&\ \ \ \ \ \cdot f_{1,2t,2t(2t-1)}(x^{-1}q^{1+k},-q^{(2t-1)(k+t)+t};q).
\end{align*}
\end{corollary}
We again point out to the reader that the expression within the parentheses in Corollary \ref{cor:main} is essentially modular, whereas the single double-sum outside of the parentheses is mixed mock modular.  Although the results \cite[Theorem 1.3, 1.4]{HM} are unable to evaluate sums of the form
\begin{equation*}
    f_{1,2t,2t(2t-1)}(x,y;q),
\end{equation*}
Theorem \ref{theo:general-hecke} can in principle do exactly that.

\begin{example} We note that for $t=2$, the set of generalized Kontsevich--Zagier series \cite[(A.1), (A.2)]{HL} and the dual $U$-functions \cite[(A.3), (A.4)]{HL} satisfy $F_{2}^{(m)}(\zeta_{N}^{-1})=U_{2}^{(m)}(-1;\zeta_{N})$, where the later read
\begin{align*}
    U_{2}^{(1)}(x;q)
    &=\sum_{n=0}^{\infty}(-xq)_{n}(-x^{-1}q)_{n}q^{n-1}
    \sum_{k=1}^{n+1}q^{k^2}\Big [ \begin{matrix} n+k\\2k-1\end{matrix} \Big ]_{q}\notag\\
    &=1+q+(x+2+x^{-1})q^2+(2x+3+2x^{-1})q^3+(3x+6+3x^{-1})q^4+\dots\\
    U_{2}^{(2)}(x;q)
    &=\sum_{n=0}^{\infty}(-xq)_{n}(-x^{-1}q)_{n}q^{n-1}
    \sum_{k=0}^{n+1}q^{k^2}\Big [ \begin{matrix} n+k+1\\2k\end{matrix} \Big ]_{q}\notag\\
    &=q^{-1}+2+(x+2+x^{-1})q+(2x+4+2x^{-1})q^2+(4x+6+4x^{-1})q^3+\dots
\end{align*}
\end{example}
Using Corollary \ref{cor:main}, where we recall that $t$ has been replaced with $t-1$, we have for $t=2$, $m=1,2$:
\begin{align*}
(1-x)U_{2}^{(1)}(-x;q)&=\frac{q^{-3}}{(q)_{\infty}^3}
\sum_{k=0}^{5}(-1)^{k}q^{\binom{k+1}{2}}\\
&\ \ \ \ \ \cdot (f_{1,11,1}(q^{k+4},q^{k-3},q)
-qf_{1,11,1}(q^{k-1},q^{k+2},q))
 \cdot f_{1,6,30}(x^{-1}q^{1+k},-q^{5k+18},q),\notag\\
(1-x)U_{2}^{(2)}(-x;q)&=\frac{q^{-4}}{(q)_{\infty}^3}
\sum_{k=0}^{5}(-1)^{k}q^{\binom{k+1}{2}}\\
&\ \ \ \ \ \cdot (f_{1,11,1}(q^{k+5},q^{k-4},q)
-q^2f_{1,11,1}(q^{k},q^{k+1},q))
 \cdot f_{1,6,30}(x^{-1}q^{1+k},-q^{5k+18},q).\notag
\end{align*}

\begin{example} We note that for $t=3$, the set of generalized Kontsevich--Zagier series \cite[(A.5), (A.6), (A.7)]{HL} and the dual $U$-functions \cite[(A.8), (A.9), (A.10)]{HL} satisfy $F_{3}^{(m)}(\zeta_{N}^{-1})=U_{3}^{(m)}(-1;\zeta_{N})$, where the later read (where a few typos have also been corrected)
\begin{align*}
    U_{3}^{(1)}(x;q)
    &=\sum_{n=0}^{\infty}\sum_{k=1}^{n+1}\sum_{j=1}^{k}
    (-xq)_{n}(-x^{-1}q)_{n}q^{n-2+k^2+j^2}
    \Big [ \begin{matrix} k+j-1\\2j-1\end{matrix} \Big ]_{q}
    \Big [ \begin{matrix} n+k+2j-1\\2k+2j-2\end{matrix} \Big ]_{q}\notag\\
    &=1+q+(x+2+x^{-1})q^2+(2x+4+2x^{-1})q^3+(4x+7+4x^{-1})q^4+\dots\\
    U_{3}^{(2)}(x;q)
    &=\sum_{n=0}^{\infty}\sum_{k=1}^{n+1}\sum_{j=0}^{k}
    (-xq)_{n}(-x^{-1}q)_{n}q^{n-2+k^2+j^2}
    \Big [ \begin{matrix} k+j\\2j\end{matrix} \Big ]_{q}
    \Big [ \begin{matrix} n+k+2j\\2k+2j-1\end{matrix} \Big ]_{q}\notag\\
    &=q^{-1}+2+(x+3+x^{-1})q+(3x+5+3x^{-1})q^2+(5x+10+5x^{-1})q^3+\dots\\
    U_{3}^{(3)}(x;q)
    &=\sum_{n=0}^{\infty}\sum_{k=1}^{n+1}\sum_{j=0}^{k}
    (-xq)_{n}(-x^{-1}q)_{n}q^{n-2+k^2+j^2}
    \Big [ \begin{matrix} k+j\\2j\end{matrix} \Big ]_{q}
    \Big [ \begin{matrix} n+k+2j+1\\2k+2j\end{matrix} \Big ]_{q}\notag\\
    &=q^{-2}+2q^{-1}+(x+3+x^{-1})+(2x+5+2x^{-1})q+(5x+8+5x^{-1})q^2+\dots
\end{align*}
\end{example}

Using Corollary \ref{cor:main}, where we recall that $t$ has been replaced with $t-1$, we have for $t=3$, $m=1$:
\begin{align*}
(1-x)U_{3}^{(1)}(-x;q)&=\frac{q^{-4}}{(q)_{\infty}^3}
\sum_{k=0}^{7}(-1)^{k}q^{\binom{k+1}{2}}\\
&\ \ \ \ \ \cdot (f_{1,15,1}(q^{k+5},q^{k-4},q)
-qf_{1,15,1}(q^{k-2},q^{k+3},q))
 \cdot f_{1,8,56}(x^{-1}q^{1+k},-q^{7k+32},q),\\
(1-x)U_{3}^{(2)}(-x;q)&=\frac{q^{-5}}{(q)_{\infty}^3}
\sum_{k=0}^{7}(-1)^{k}q^{\binom{k+1}{2}}\\
&\ \ \ \ \ \cdot (f_{1,15,1}(q^{k+6},q^{k-5},q)
-q^2f_{1,15,1}(q^{k-1},q^{k+2},q))
 \cdot f_{1,8,56}(x^{-1}q^{1+k},-q^{7k+32},q),\\
(1-x)U_{3}^{(3)}(-x;q)&=\frac{q^{-6}}{(q)_{\infty}^3}
\sum_{k=0}^{7}(-1)^{k}q^{\binom{k+1}{2}}\\
&\ \ \ \ \ \cdot (f_{1,15,1}(q^{k+7},q^{k-6},q)
-q^3f_{1,15,1}(q^{k},q^{k+1},q))
 \cdot f_{1,8,56}(x^{-1}q^{1+k},-q^{7k+32},q).
\end{align*}

\begin{proof}[Proof of Lemma \ref{lem:thetaToHecke}] We look at $(r,s)\rightarrow (2r,2s+1), \ (2r+1,2s)$.  For the first we have
{\allowdisplaybreaks \begin{align*}
\Big ( \sum_{r,s \ge0}&-\sum_{r,s <0}\Big )(-1)^{r+s+1}
q^{\frac{1}{2}r^2+\frac{4t-1}{2}r(2s+1)+\frac{1}{8}(2s+1)^2
+(p+m)r+\frac{p-m}{2}(2s+1)}\\
&=-\Big ( \sum_{r,s \ge0}-\sum_{r,s <0}\Big )(-1)^{r+s}q^{\frac{1}{2}r^2+(4t-1)rs+\frac{4t-1}{2}r
+\frac{1}{2}s^2+\frac{1}{2}s+\frac{1}{8}
+(p+m)r+(p-m)s+\frac{p-m}{2}}\\
&=-\Big ( \sum_{r,s \ge0}-\sum_{r,s <0}\Big )(-1)^{r+s}q^{\binom{r}{2}+(4t-1)rs
+\binom{s}{2}
+(p+m+2t)r+(p-m+1)s+\frac{1}{8}+\frac{p-m}{2}}\\
&=-q^{\frac{1}{8}+\frac{p-m}{2}}f_{1,4t-1,1}(q^{p+m+2t},q^{p-m+1},q).
\end{align*}}%
For the second we have
{\allowdisplaybreaks \begin{align*}
\Big ( \sum_{r,s \ge0}&-\sum_{r,s <0}\Big )(-1)^{r+s}
q^{\frac{1}{2}r^2+\frac{1}{2}r+\frac{1}{8}+\frac{4t-1}{2}(2r+1)s+\frac{1}{2}s^2
+(p+m)r+\frac{p+m}{2}+(p-m)s}\\
&=\Big ( \sum_{r,s \ge0}-\sum_{r,s <0}\Big )(-1)^{r+s}q^{\binom{r}{2}+r+\frac{1}{8}+(4t-1)rs+\frac{4t-1}{2}s
+\binom{s}{2}+\frac{1}{2}s+(p+m)r+\frac{p+m}{2}+(p-m)s}\\
&=\Big ( \sum_{r,s \ge0}-\sum_{r,s <0}\Big )(-1)^{r+s}q^{\binom{r}{2}+(4t-1)rs
+\binom{s}{2}+(p+m+1)r+(p-m+2t)s+\frac{1}{8}+\frac{p+m}{2}}\\
&=q^{\frac{1}{8}+\frac{p+m}{2}}f_{1,4t-1,1}(q^{p+m+1},q^{p-m+2t},q).\qedhere
\end{align*}}%
\end{proof}

\section{Proof of a theta-less identity}\label{section:proof-thetaless}
Here we give a theta-less expression for the Hecke-type double sum.
As in section \ref{section:proof-hecke} we first assume $|q|^{D} < |x_1^cy_1^{-b}x_2|< 1$.
We consider
\[\begin{split}
\sum_{r_1,s_1,r_2\in\Z} & \sg(r_1,s_1)\, q^{\frac12a r_1^2+br_1s_1+\frac12cs_1^2+\frac12 cDr_2^2}\, x_1^{r_1}y_1^{s_1} x_2^{r_2}\\
&=\sum_{r_2\in\Z} q^{\frac12 cDr_2^2}\, x_2^{r_2} \sum_{r_1,s_1\in\Z} \sg(r_1,s_1)\, q^{\frac12a r_1^2+br_1s_1+\frac12cs_1^2}\, x_1^{r_1}y_1^{s_1}
\end{split}\]
and split $\sg(r_1,s_1)$ as 
\[ \sg(r_1,s_1)=\sg(r_1,cr_2-r_1)+\sg(s_1,-cr_2+r_1-1).\]
In the section \ref{section:proof-hecke} we have already computed the part coming from $\sg(r_1,cr_2-r_1)$:
\[\sum_{t=0}^{c-1} \sum_{s_1\in\Z} q^{\frac12a t^2+bts_1+\frac12cs_1^2}x_1^t y_1^{s_1}\sum_{r_2\in\Z}\frac{q^{\frac12 cDr_2^2} x_2^{r_2}}{1-x_1^cy_1^{-b}x_2 q^{D(cr_2-t)}}\]
We now consider the part coming from $\sg(s_1,-cr_2+r_1-1)$:
we substitute $r_1+cr_2$ for $r_1$ and rewrite the exponent of $q$ to get
\[\begin{split}
    \sum_{r_1,s_1,r_2\in\Z} & \sg(s_1,-cr_2+r_1-1)\, q^{\frac12a r_1^2+br_1s_1+\frac12cs_1^2+\frac12 cDr_2^2}\, x_1^{r_1}y_1^{s_1} x_2^{r_2}\\
    &= \sum_{r_1,s_1,r_2\in\Z} \sg(s_1,r_1-1)\, q^{\frac12\frac{aD} {b^2}r_1^2+\frac Db r_1s_1+\frac12 \frac c{b^2} (b^2r_2+ar_1+bs_1)^2} x_1^{r_1+cr_2}y_1^{s_1}x_2^{r_2}.
\end{split}\]
We write $r_1$ as $b^2k+u$ where $u$ is in a complete residue system modulo $b^2$ and $k\in\Z$.
Further, we write $s_1$ as $bl+v$ with $0\leq v<b$ and $l\in\Z$.
We substitute $r_2-ak-l$ for $r_2$ to get
\[ \begin{split}
\sum_{u\smod{b^2}} &\sum_{v=0}^{b-1} q^{\frac12 au^2+buv+\frac12 cv^2} x_1^uy_1^v \sum_{r_2\in\Z} q^{\frac12 b^2cr_2^2+c(au+bv)r_2}(x_1^cx_2)^{r_2} \\
&\cdot\sum_{k,l\in\Z} \sg(bl+v,b^2k+u-1)\, q^{\frac12 ab^2Dk^2+(au+bv)Dk}(x_1^Dx_2^{-a})^k \Bigl( x_1^{-c}y_1^b x_2^{-1} q^{D(b^2k+u)}\Bigr)^l.
\end{split}\]
Using $\sg(bl+v,b^2k+u-1)=\sg(l,b^2k+u-1)$ and \eqref{eq:geom}  with $q$ replaced by $q^{D}$, $n=b^2k+u-1$ and $z=x_1^{-c}y_1^{b}x_2^{-1}q^D$ we then obtain
\[ \sum_{u\smod{b^2}} \sum_{v=0}^{b-1} q^{\frac12 au^2+buv+\frac12 cv^2} x_1^uy_1^v \sum_{r_2\in\Z} q^{\frac12 b^2cr_2^2+c(au+bv)r_2}(x_1^cx_2)^{r_2} \sum_{k\in\Z} \frac{q^{\frac12 ab^2Dk^2+(au+bv)Dk}(x_1^Dx_2^{-a})^k}{1- x_1^{-c}y_1^b x_2^{-1} q^{D(b^2k+u)}}.\]
As in section \ref{section:proof-hecke} we combine the different parts, use the identity theorem and divide by 
\[ \sum_{r_2\in\Z} q^{\frac12 cDr_2^2}\, x_2^{r_2}\]
to get
\begin{theorem}
Let $a,b,c$ be positive integers such that $D:=b^2-ac>0$.  For generic $x_1$, $x_2$, $y_1$ we have:
\[\begin{split}
\sum_{r_1,s_1\in\Z} &\sg(r_1,s_1)\, q^{\frac12a r_1^2+br_1s_1+\frac12cs_1^2}\, x_1^{r_1}y_1^{s_1}\\
&=\sum_{t=0}^{c-1} \sum_{s_1\in\Z} q^{\frac12a t^2+bts_1+\frac12cs_1^2}x_1^t y_1^{s_1}\Bigl(\sum_{r_2\in\Z} q^{\frac12 cDr_2^2}\, x_2^{r_2}\Bigr)^{-1}\sum_{r_2\in\Z}\frac{q^{\frac12 cDr_2^2} x_2^{r_2}}{1-x_1^cy_1^{-b}x_2 q^{D(cr_2-t)}}\\
&\qquad +\Bigl(\sum_{r_2\in\Z} q^{\frac12 cDr_2^2}\, x_2^{r_2}\Bigr)^{-1} \sum_{u\smod{b^2}} \sum_{v=0}^{b-1} q^{\frac12 au^2+buv+\frac12 cv^2} x_1^uy_1^v\\
&\qquad \qquad \cdot \sum_{r_2\in\Z} q^{\frac12 b^2cr_2^2+c(au+bv)r_2}(x_1^cx_2)^{r_2} \sum_{k\in\Z} \frac{q^{\frac12 ab^2Dk^2+(au+bv)Dk}(x_1^Dx_2^{-a})^k}{1- x_1^{-c}y_1^b x_2^{-1} q^{D(b^2k+u)}}
\end{split}\]
\end{theorem}

Note that
\[ \sum_{k\in\Z} \frac{q^{\frac12 ab^2Dk^2+(au+bv)Dk}(x_1^Dx_2^{-a})^k}{1- x_1^{-c}y_1^b x_2^{-1} q^{D(b^2k+u)}}\]
actually is a level $a$ Appell function.
However, it can easily be written in terms of level 1 Appell functions and the function $m$, see for example Lemma 2 in \cite{Zw2}.

\section{A derivation of Identity (\ref{equation:f123-mxzq})} \label{section:derive-f123}

We can express the universal mock theta function $g_3(x;q)$ \cite{H1,H2} in terms of Appell functions \cite[Proposition $4.2$]{HM}:
\begin{equation}
g_3(x;q)=-x^{-1}m(q^2x^{-3},x^2;q^3)-x^{-2}m(qx^{-3},x^2;q^3),
\label{equation:gm-spec}
\end{equation}
where we have the easily shown
\begin{align}
g_3(1/x;q)&=g_3(qx;q),\label{equation:g-inv}\\
g_3(qx;q)&=-x-x^2-x^3g_3(x;q).\label{equation:g-func}
\end{align}

For $t=1$, the modularity of $(1-x)U_t(-x;q)$ is known \cite{M1}.  For $t=1$, identity (\ref{equation:gen-utm}) specialises to
\begin{equation*}
U_1(x;q)=\sum_{n= 0}^{\infty}(-xq)_n(-q/x)_nq^n.
\end{equation*}
From \cite{L}, we have (slightly rewritten) that 
\begin{equation}
1+\Big ( 1-x\Big )\sum_{n=0}^{\infty}q^{n+1}(qx)_n(q/x)_n
=-\frac{x^{-1}q^3}{(q)_{\infty}}f_{1,2,3}(x^{-1}q^3,q^6;q),\label{equation:Love-equiv}
\end{equation}
or equivalently \cite[(6.5)]{HM} that
\begin{equation*}
\Big ( 1-x\Big )\sum_{n=0}^{\infty}q^{n}(qx)_n(q/x)_n
=\frac{1}{(q)_{\infty}}f_{1,2,3}(x^{-1}q^2,q^4;q).
\end{equation*}
From \cite[$(4.9)$, $(4.7)$]{M1} we have (slightly rewritten) that 
\begin{align*}
1+\Big ( 1-x\Big )\sum_{n=0}^{\infty}q^{n+1}(q/x)_n(xq)_n
&= -\frac{1}{x}g_3(x^{-1};q)+\frac{\Theta(x^{-1};q)}{(q)_{\infty}}m(x^{-2},x;q)\\
&= -\frac{1}{x}g_3(qx;q)-x\frac{\Theta(x;q)}{(q)_{\infty}}m(x^{2},x^{-1};q).
\end{align*}
where for the last equality we have used (\ref{equation:g-inv}), (\ref{equation:j-elliptic}), and \cite[(3.2c)]{HM}.  Thus we can write
{\allowdisplaybreaks \begin{align*}
&(1-x ) U_1(-x;q)\\
&= -\frac{1}{q}-\frac{1}{qx}g_3(qx;q)-\frac{x}{q}\frac{\Theta(x;q)}{(q)_{\infty}}m(x^{2},-q;q) \\
&=-\frac{1}{q}+\frac{1}{qx}\Big [ \frac{1}{qx}m(q^{-1}x^{-3},q^2x^2;q^3)+\frac{1}{(qx)^2}m(q^{-2}x^{-3},q^2x^2;q^3)\Big ]
-\frac{x}{q}\frac{\Theta(x;q)}{(q)_{\infty}}m(x^{2},x^{-1};q)\\
&= \frac{1}{(qx)^2}m(q^{-1}x^{-3},q^2x^2;q^3)-\frac{1}{q}+\frac{1}{(qx)^3}m(q^{-2}x^{-3},q^2x^2;q^3)
-\frac{x}{q}\frac{\Theta(x;q)}{(q)_{\infty}}m(x^{2},x^{-1};q)\\
&=q^{-2}x^{-2}m(q^{-1}x^{-3},q^2x^2;q^3)-q^{-1}m(qx^{-3},q^2x^2;q^3)
-\frac{x}{q}\frac{\Theta(x;q)}{(q)_{\infty}}m(x^{2},x^{-1};q),
\end{align*}}%
where we have used (\ref{equation:gm-spec}) for the third line and \cite[(3.2c)]{HM} in the fifth line.   This gives (\ref{equation:f123-mxzq}).

\section*{Acknowledgements}
This research was supported by the Theoretical Physics and Mathematics Advancement Foundation BASIS, agreement No. 20-7-1-25-1.


\begin{thebibliography}{999999}

\bibitem{BOPR} J. Bryson, K. Ono, S. Pitman, R. C. Rhoades, {\em Unimodal sequences and quantum and mock modular forms}, Proc. Natl. Acad. Sci. USA, {\bf 109} (2012), no. 40, 16063--16067.

\bibitem{DMZ} A. Dabholkar, S. Murthy, D. B. Zagier, {\em Quantum black holes, wall crossing, and mock modular forms}, arXiv:1208.4074.

\bibitem{H1} D. R. Hickerson, {\em A proof of the mock theta conjectures}, Invent. Math. {\bf 94} (1988), no. 3,  639--660.

\bibitem{H2} D. R. Hickerson, {\em On the seventh order mock theta functions}, Invent. Math. {\bf 94} (1988), no. 3,  661--677.

\bibitem{HM} D. R. Hickerson, E. T. Mortenson, {\em Hecke-type double sums, Appell--Lerch sums, and mock theta functions, I}, Proc. London Math. Soc. (3) {\bf 109} (2014), no. 2, 382--422. 

\bibitem{Hik1} K. Hikami, {\em Difference equation of the colored Jones polynomial for the torus knot}, Int. J. Math. {\bf 15} (2004), pp. 959--965.

\bibitem{Hik2} K. Hikami, {\em $q$-series and $L$-functions related to half-derivatives of the Andrews--Gordon identity}, Ramanujan J. {\bf 11} (2006), pp. 175--197.

\bibitem{HL} K. Hikami, J. Lovejoy, {\em Torus knots and quantum modular forms}, Res. Math. Sci. (2015) {\bf{2}}:2 (29 January 2015)

\bibitem{WBRL} W. B. R. Lickorish, {\em An Introduction to Knot Theory}, vol. 175 of Graduate Texts in Mathematics, Springer, New York, 1997.

\bibitem{L} J. Lovejoy, {\em Ramanujan-type partial theta identities and conjugate Bailey pairs}, Ramanujan J. {\bf 29} (2012), no. 1-3, 51--67.

\bibitem{M1} E. T. Mortenson, {\em On the dual nature of partial theta functions and Appell--Lerch sums}, Adv. Math. {\bf 264} (2014), 236--260.

\bibitem{Za2001} D. B. Zagier, {\em Vassiliev invariants and a strange identity related to the Dedekind eta-function}, Topology {\bf 40} (2001), no. 5, 945-960.

\bibitem{Za2020} D. B. Zagier, {\em Quantum modular forms}, in Quanta of Maths: Conference in honor of Alain Connes, Clay Mathematics Proceedings {\bf 11}, AMS and Clay Mathematics Institute $2010$, 658--675.

\bibitem{Zw} S. Zwegers, {\em Mock theta functions}, Ph.D. Thesis, Universiteit Utrecht, 2002.

\bibitem{Zw2} S. Zwegers, {\em Multivariable Appell functions and nonholomorphic Jacobi forms}, Res. Math. Sci. {\bf 6} (2019)

\end{thebibliography}
\end{document}